\documentclass[
  a4paper, 
  reqno, 
  oneside, 
  11pt
]{amsart}

\usepackage[dvipsnames]{xcolor} 
\colorlet{cite}{LimeGreen!50!Green}
\usepackage{tikz}  
\usepackage{rotating}
\usetikzlibrary{arrows,positioning}
\tikzset{ 
  baseline=-2.3pt,
  text height=1.5ex, text depth=0.25ex,
  >=stealth,
  node distance=2cm,
  mid/.style={fill=white,inner sep=2.5pt},
}

\usepackage{lmodern}
\usepackage{tikz-cd}
\usepackage{ amssymb, amsfonts}

\usepackage{graphicx,caption,subcaption}
\usepackage{fourier}
\usepackage{braket}  
\usepackage{microtype}

\usepackage[dvipsnames]{xcolor}
\usepackage[%
  bookmarks=true,			
  unicode=true,			
  pdftitle={mirrrors},		%
  pdfauthor={Elizabeth Gasparim}{},	%
  pdfkeywords={}{adjoint orbit}{compactification}{Lefschetz fibration},	
  colorlinks=true,		
  linkcolor=Blue,			
  citecolor=cite,		
  filecolor=magenta,		
  urlcolor=RoyalBlue			
]{hyperref}				
\usepackage{cleveref}
\newcommand{\ce}{\mathrel{\mathop:}=}

\theoremstyle{plain}	
\newtheorem{theorem}{Theorem} 
\newtheorem{lemma}[theorem]{Lemma} 

\newtheorem*{theorem*}{Theorem}

\newtheorem*{conjecture*}{Conjecture}
\theoremstyle{remark}
\newtheorem{remark}[theorem]{Remark}

\newtheorem{example}[theorem]{Example}

\DeclareMathOperator{\Ext}{Ext}
\DeclareMathOperator{\Hom}{Hom}
\DeclareMathOperator{\coker}{coker}
\DeclareMathOperator{\im}{im}
\DeclareMathOperator{\Fuk}{Fuk}
\DeclareMathOperator{\Tot}{Tot}

\DeclareMathOperator{\Coh}{Coh}

\DeclareMathOperator{\Spec}{Spec}
\DeclareMathOperator{\Diag}{Diag}
\DeclareMathOperator{\tr}{tr}
\DeclareMathOperator{\adj}{adj}
\DeclareMathOperator{\LG}{LG}

\begin{document}
\author{Elizabeth Gasparim }
\address{E. Gasparim- Depto. Matem\'aticas, Univ. Cat\'olica del Norte, Antofagasta, Chile. etgasparim@gmail.com}
\title{Intrinsic mirrors for minimal adjoint orbits}

\begin{abstract} I discuss   mirrors
of   Landau--Ginzburg  models  formed by 
 a  minimal semisimple adjoint orbit of $\mathfrak{sl}(n)$  together with  a potential  obtained 
 via the Cartan--Killing form.
I show that the  Landau--Ginzburg models produced by the Gross--Siebert recipe
give precisely  the objects  of the desired mirrors. 

It is known that   Landau--Ginzburg  model $\LG(2)$   over the 
 semisimple adjoint orbit of $\mathfrak{sl}(2)$ does not have projective mirrors. 
 I prove Homological Mirror Symmetry for  $\LG(2)$  by constructing
  a  Landau--Ginzburg mirror and  showing that its Orlov category of singularities 
 is equivalent to $\Fuk(\LG(2))$.
\end{abstract}
\maketitle

\tableofcontents

\section{LG models on adjoint orbits,  Fukaya--Seidel categories and mirrors}

Given a nicely behaved  $A$-side  Landau--Ginzburg model, that is, a complex potential function  $f\colon X\rightarrow \mathbb C$ with 
at most Morse type singularities defining a symplectic Lefschetz fibration on a noncompact symplectic manifold $X$, there
are 2 main
conjectured  ways to find a mirror. \\

{\bf HMS1}.  Compactify  $X$ and extend  its potential to a function with target $\mathbb P^1$, and then look for  (a subvariety of) a
complex variety $Y$ such that the bounded derived category of coherent sheaves $D^b\! \Coh(Y) $
 is equivalent to the Fukaya--Seidel category of the compactification  of $(X,f)$. Hence, here we search for an equivalence of categories
$$\Fuk (X,f) \equiv D^b\!\Coh(Y)  .$$

{\bf HMS2}. find a mirror Landau--Ginzburg model $g \colon Y \rightarrow \mathbb C$ such that the Orlov's category of singularities 
$D_{sg}(Y) \ce \bigoplus_i \frac{D^b\! \Coh (Y_i)}{\mathfrak{Perf}(Y_i)} $, where $Y_i$ are the singular fibres of $g$,
is equivalent to the  Fukaya--Seidel category  $\Fuk(X,f)$ of Lagrangian thimbles of $f$. 
Hence, here we search for an equivalence of categories
$$\Fuk (X,f) \equiv  D_{sg}(Y). $$

We consider  the case when $X$ is a minimal adjoint orbit of $\mathfrak{sl}(n, \mathbb C)$, namely, 
the adjoint orbit   $\mathcal O(H_0)$ of an element $H_0= \Diag(n-1, -1, \dots, -1)$, which is diffeomorphic 
to the cotangent bundle of $\mathbb P^{n-1}$
 \cite[Thm.\thinspace 2.1] {GGSM2}.
Over such minimal adjoint orbits we have a preferred choice, denoted $\LG(n)$,  of Landau--Ginzburg model coming 
from Lie theory. Indeed, for each regular element $H \in \mathfrak h_{\mathbb{R}}$
the theorem 	\cite[Thm.\thinspace3.1]{GGSM1} tells us that 
	 the potential $f_{H}:\mathcal{O} \left( H_{0}\right) \rightarrow \mathbb{C}$ defined by
	\[
	f_{H}\left( A\right) = \langle H,A\rangle \qquad A\in \mathcal{O}\left(H_{0}\right)
	\]
	has a finite number of isolated singularities and defines a symplectic Lefschetz fibration, hence an $A$-side Landau--Ginzburg model; 
	that is to say
that  $f_H$ admits only  nondegenerate singularities; 
 there exists a symplectic form $\Omega $ on $\mathcal{O}\left(H_{0}\right) $ such that the regular fibres 
 of $f_H$ are pairwise isomorphic symplectic submanifolds;
	and  at each critical point, the tangent cone of singular fibre can be written as a disjoint union of affine subspaces contained in $\mathcal O \left( H_0 \right)$, each symplectic with respect to $\Omega$.

 \cite{BBGGSM} showed that conjecture HMS1 totally fails  to produce a mirror for  the Landau--Ginzburg model $\LG(2)$
associated to the semisimple adjoint orbit of $\mathfrak{sl}(2, \mathbb C)$, by proving that there does not exist  (any subvariety of ) a projective variety $Y$ 
whose bounded derived category of coherent sheaves $D^b\!\Coh(Y)$ is equivalent to the Fukaya--Seidel category $\Fuk
\LG(2)$. 

 I show that the Gross--Siebert intrinsic mirror symmetry algorithm produces a Landau--Ginzburg model $(\mathcal Y_{n+1},g)$ 
mirror candidate  to $\LG(n)$ for each $n$, in the sense that the Orlov's category of singularities $D_{sg}(\mathcal Y_{n+1})$ corresponds  to 
$\Fuk\LG(n)$.  I prove the full categorical equivalence for the case of $\mathfrak{sl}(2)$.

 \begin{theorem} \label{mirror2} Homological Mirror Symmetry for $LG(2)$ works as follows:
 \begin{itemize} 
 \item {\bf HMS 1} fails for $LG(2)$, that is, for any (quasi) projective variety $Y$:
 $$ \Fuk(\LG(2))\not\equiv D^b\!\!\Coh(Y)$$ 
\item {\bf HMS 2} holds true  for $LG(2)$, that is, 
 we have an equivalence of categories:
 $$\Fuk(\LG(2)) \equiv D_{sg}(Y_2) .$$
 \end{itemize}
 \end{theorem}
 
 For the cases when  $n>2$, I discuss 
the equivalence $\Fuk(\LG(n+1)) \equiv D_{sg}(\mathcal Y_{n+1})$ at the level of objects and  calculate the morphisms in the 
category
 $D_{sg}(\mathcal Y_{n+1})$. 
Since the morphisms of $\LG(n)$  are not yet known for   $n>2 $, at this point   
we can just say that $(\mathcal Y_{n+1},g)$ is a mirror candidate for $\LG(n)$.

\begin{theorem}\label{mirrorn}
The intrinsic mirror symmetry algorithm produces an $\LG$-model $(\mathcal Y_{n+1},g)$
for which
 $$ \Fuk(\LG(n+1)) \simeq D_{sg}(\mathcal Y_{n+1})$$
 is a 1-1 correspondence of objects.
\end{theorem}

I also  completely calculate the  category $D_{sg}(\mathcal Y_{n+1})$. Its objects
are the nonperfect sheaves  $\mathcal F(z_0), \mathcal F(z_1),\dots, \mathcal F (z_n)$
 defined in section \ref{catsing}, where $\mathcal F(z_0)$ is supported on the fibre over zero, 
 and all the others are supported on the fibre over infinity. 
 This shows that the general case is analogous to what happens for $LG(3)$ 
 as depicted in figure \ref{fig2}.

\begin{theorem}\label{ort} The Orlov category of singularities  of the Landau--Ginzburg model $({\mathcal Z}_{n+1})$
is generated by the nonperfect shaves   $\mathcal F(z_0), \mathcal F(z_1),\dots, \mathcal F (z_n)$
with  morphisms: 
$${\mathbf \Hom}\left(\mathcal F(z_i), \mathcal F(z_j)\right) = 
\left\{\begin{array}{lllll}
N_{ij} \bigoplus_{t=2s+1} M_{ij}[t]  & \text{if} & i \cdot j \neq 0,\\
0 & \text{if} & i\cdot j =0 \text{.}
\end{array}
\right.$$ 
\end{theorem}

The calculation of  morphisms   in $D_{sg}(\mathcal Y_{n+1})$  may
 be regarded as a prediction of how such morphisms in $\LG(n)$ ought to behave.

\section{Minimal orbits and extending the potential to the compactification}

We  focus on the case of minimal semisimple orbits. Hence, for each $n$,  we discuss the orbit 
through  $H_0=\Diag(n,-1, \dots, -1)$ which  compactifies to 
$\mathbb P^n \times {\mathbb P^n}^*.$

Let $H=\Diag(\lambda_1,\ldots,\lambda_{n+1})\in \mathfrak{h}_{\mathbb R}$, with $\lambda_1>\ldots>\lambda_{n+1} $ and
 $\lambda_1 + \ldots +\lambda_{n+1}=0$ (where $\mathfrak{h}$ is the  Cartan subalgebra of  $\mathfrak{sl}(n+1)$).

Following \cite{BGGSM}, we  describe a rational map (factored through the  Segre embedding) that coincides with the potential $f_H$ 
on the adjoint orbit $\mathcal O(H_0)$. Such a rational map is given by 
\begin{equation}
 \psi:\mathbb{P}^{n}\times G_{n}(\mathbb{C}^{n+1})\rightarrow \mathbb{P}^1, 
\end{equation} 
\begin{equation}
\psi([v],[\varepsilon]) = \frac{\tr((v\otimes \varepsilon) \rho(H))}{\tr(v\otimes \varepsilon)}=\frac{\sum_{i=1}^{n+1}{\lambda_i a_{i1} (\adj g)_{1i} }}{\sum_{i=1}^{n+1}{ a_{i1} (\adj g)_{1i} }},
\end{equation}
where the identification $([v],[\varepsilon])\mapsto v\otimes \varepsilon$ is described in \cite[Sec. 4.2]{GGSM2}
and $\rho$
is the canonical representation . Observe that if  $([v],[\varepsilon])$ belongs to the adjoint orbit, then $\tr(v\otimes\varepsilon)=1$. 
Furthermore, the complement of the orbit in the compactification  is the  { incidence correspondence variety} $\Sigma$, that is, 
the set of pairs $(\ell,\pi)$ such that
$
0\subset \ell\subset \pi \subset\mathbb{C}^{n+1},
$
where $\pi$ is a hyperplane in  $\mathbb{C}^{n+1}$ and $\ell\subset \pi$ is a line. 
The variety $\Sigma$ is the 2 step flag manifold classically denoted in Lie theory by  $\mathbb{F}(1,n)$;
it is a divisor in $\mathbb P^n \times \mathbb P^n$, and the fibre at infinity of the potential $f_H$,
that is, $\psi=[f_H:1]$ over the adjoint orbit, and $\psi(\mathbb F(1,n))=[1:0]$.

\section{The Landau--Ginzburg  model $\LG(2)$}
We start with the Landau--Ginzburg model $\LG(2) = (\mathcal O_2, f_H)$ where 
$\mathcal O_2$ is the semisimple adjoint orbit of $\mathfrak{sl}(2,\mathbb C)$, 
together with the potential $f_H$ obtained by pairing with the regular element $H$ via the Cartan--Killing form.
We choose $$H=H_0= \left(\begin{matrix} 1 & 0 \\ 0 & -1 \end{matrix}\right)$$
and the adjoint orbit  $\mathcal O_2\ce \mathcal O(H_0)$ is then
$\mathcal O_2= \mathrm{Ad}(\mathrm{SL}(2)) H_0= \left\{gH_0g^{-1}: g \in \mathrm{SL}(2)\right\}$
with potential 
$$\begin{array}{rcl} f_H\colon \mathcal O_2& \longrightarrow & \mathbb C\\
 A & \longmapsto & \langle H,A\rangle.\end{array}$$

The Fukaya--Seidel category of $\LG(2)$ was calculated in \cite{BBGGSM} as follows:

\begin{theorem*}\cite[Thm.\thinspace 3.1]{BBGGSM}
The Fukaya--Seidel category of $\LG(2)$ is generated by 2 Lagrangians $L_0,L_1$ with morphisms:
$$\Hom(L_i,L_j) = \left\{\begin{array}{ll}
\mathbb Z \oplus \mathbb Z[-1] & i<j\\
\mathbb Z & i=j\\
0 & i>j
\end{array}\right.$$
and the products  $m_k$ all vanish except for $m_2(.,id)$ and $m_2(id, .)$.

\end{theorem*}

\begin{remark}
We observe that the 2 Lagrangians $L_0$ and $L_1$ have  the same vanishing cycle, thus forming
a matching cycle, as depicted in figure \ref{sphere}. The matching cycle formed by $L_0 \cup L_1$ 
is in fact the Lagrangian sphere $S^2\simeq \mathbb P^1$ identified with the zero section of $T^*\mathbb P^1$
(the smooth type of $\mathcal O_2$). However, there will be no matching cycles on $\mathcal O_n$ 
for $n>2$, because the existence of 
 a Lagrangian $n$-sphere is prohibited by the topological
type  $\mathcal O_n \simeq T^*\mathbb P^n$ when $n>1$.
This explains the absence of spheres in figure \ref{fig2}.
\end{remark}

The goal is  to calculate a Landau--Ginzburg mirror to $\LG(2)$. 
For this, we  use the compactification described in \cite[Thm.\thinspace 6.3]{BBGGSM}
giving  $\overline{\LG}(2)=(\mathbb P^1\times \mathbb P^1,R_H)$
where $$R_H([x:y][z:w]) = [xw+yz:xw-yz]$$
is the extension of $f_H$ to a rational map $R_H\colon \mathbb P^1\times \mathbb P^1\dashrightarrow \mathbb P^1$. 
Let $\Delta $ be the diagonal in $\mathbb P^1\times \mathbb P^1$, then we have  that
$\mathbb P^1\times \mathbb P^1\simeq \mathcal O_2 \cup \Delta .$

\subsection{The intrinsic mirror of $\LG(2)$}\label{int2}

We start with the pair $(\mathbb P^1\times \mathbb P^1,\Delta)$, but to
fit into the hypothesis of  log geometry we need 
to work relatively to an anti-canonical divisor, in this case a divisor of
type $(2,2)$. So, we choose  an additional divisor 
$\Delta'$ of type $(1,1)$ such that $\Delta \cap \Delta'= \{p_1,p_2\}$ consists of 2 points. 

Let $\widetilde{X}$ denote the blow up of $\mathbb P^1\times \mathbb P^1$
at these 2 points.
We then study the pair 
$$(\widetilde{X},D) \qquad \text{with}  \qquad D= D_1+D_2+ D_3 + D_4$$
where $D_2$ and $D_4$ are the proper transforms of 
$\Delta'$  and $\Delta$ respectively, and $D_1$ and $D_3$ are the exceptional curves obtained from blowing
up the 2 points; thus
$D_1^2=D_3^2=-1.$

To obtain the dual picture, following the Gross--Siebert intrinsic mirror recipe,
we consider the vector space generated by the divisors $D_i$; then
 take a dual basis $D_1^*,D_2^*,D_3^*,D_4^*$, 
and theta functions $\vartheta_0,\vartheta_1,\vartheta_2,\vartheta_3,\vartheta_4$.
Here $\vartheta_0$ is the theta function corresponding to the
origin in the dual complex. 

To discuss punctured Gromov--Witten invariants, we choose
$$p=D_2^*, \quad q=D_4^*, \quad r=0.$$
Following the  steps described in the introduction of  \cite{GS}, we obtain a dual surface $S$ described by the 2 equations:
\begin{equation}\label{eq1}\vartheta_2\vartheta_4 = \vartheta_0(t^M+t^N)+\vartheta_1t^{D_1}+\vartheta_3t^{D_3}
\end{equation}
\begin{equation}\label{eq2}\vartheta_1\vartheta_3= \vartheta_0t^{D_2},
\end{equation}
where $M$ and $N$ are the classes of lines of type $(1,0)$ and $(0,1)$ respectively.
These equations arise as follows. In the product
$\vartheta_2\vartheta_4$, we are looking for curves which meet 
$D_2$ and $D_4$ to order $1$ at one point each, but may have a negative
tangency point with one of the $D_i$. 
First, in a ruling of  $\mathbb P^1 \times \mathbb P^1$  either class
 $M$ or $N$ meets $D_2$ and $D_4$ transversally at one point
each, and is disjoint from $D_1$ and $D_3$. This produces the 
contribution $\vartheta_0(t^M+t^N)$. Second, the curve $D_1$
meets $D_2$ and $D_4$ transversally at one point each, and has self-intersection
$-1$ with $D_1$. Hence, after fixing a point $z\in D_1$, there is
a unique punctured curve structure on $D_1$ giving the term
$\vartheta_1t^{D_1}$. Similarly, replacing $1$ with $3$, we get
$\vartheta_3t^{D_3}$. These are the only contributions to
the product $\vartheta_2\vartheta_4$.
For $\vartheta_1\vartheta_3$, a curve in the pencil $|D_2|$ meets
$D_1$ and $D_3$ transversally, so after choosing a general point $z$, there
is one such curve passing through $z$. This gives the term
$\vartheta_0 t^{D_2}$.

Given that $\vartheta_0$ is the theta function corresponding to the
origin in the dual complex, it is known that $\vartheta_0$ is the
unit in the dual ring, so we may replace it with $1$ in the equations that follow. 

Let $P$ be the smallest monoid containing all classes of stable maps into  $S$.
We then obtain a mirror,  by first taking
$$R= \mathbb C[P][\vartheta_1,\vartheta_2,\vartheta_3,\vartheta_4]/{(\ref{eq1})(\ref{eq2})}$$
and then the mirror variety is given by  
$\Spec R .$ 
This would give the mirror to $\widetilde{X} \setminus D_2$, since we have added the extra divisor. 

To continue analyzing the equations, we use the argument that $\vartheta_0$ is a unit, 
reducing the equations to:
$$\vartheta_2\vartheta_4 = \vartheta_0(t^M+t^N)+\vartheta_1t^{D_1}+\vartheta_3t^{D_3},
$$
where we have
replaced $\vartheta_0$ with $1$. We then note that the group of curve
classes of the variety is generated by $M,N,D_1,D_3$, and
$D_2\sim M+N-D_1-D_3\sim D_4$. Thus, setting $t^M,t^N,t^{D_1},t^{D_3}$
to coefficients $a,b,c,d$ would give equations
\begin{align*}
\vartheta_2\vartheta_4 = {} & a+b+c\vartheta_1 +d \vartheta_3\\
\vartheta_1\vartheta_3 = {} & {ab\over cd},
\end{align*}
where we eliminate $\vartheta_3$, getting
\[
\vartheta_2\vartheta_4=a+b+c\vartheta_1+{ab\over c}\vartheta_1^{-1}.
\]
Then, for a sufficiently general choice of coefficients $\alpha,\beta,
\gamma$, 
we may rewrite the general equation as
\begin{equation}\label{mir}
\vartheta_2\vartheta_4=\alpha\vartheta_1+\beta\vartheta_1^{-1}+\gamma
\end{equation}
which defines a surface in $ \mathbb C^* \times \mathbb C^2 $.

To obtain an LG model,  we take 
 as a potential the function  $\vartheta_2$,
linked to the compactification divisor, and 
 we shall see that the choice of $\vartheta_2$ solves the problem at hand.
 

\subsection{Critical points  of the potential}

To simplify calculations we set the constants to one: $\alpha=\beta= \gamma=1$,
and after one blow-up, we obtain the equation of the  surface  $u\vartheta_2=v
(\vartheta_1+\vartheta_1^{-1}+1 ).$  Changing to more standard algebraic coordinates
 $ x \ce \vartheta_1, y \ce \vartheta_2$ and compactifying the direction $\vartheta_4$ to $\mathbb P^1$ 
 with coordinates $[u:v]$,
our problem is then to study the surface 
$\mathcal Y_2 \subset \mathbb C^*\times \mathbb C\times \mathbb P^1 = \{(x,y),[u:v]\}$  given by
$$\mathcal Y_2\ce \{uy=v(x+1+1/x)\}$$
with potential $$g\ce y=v(x+1+1/x)/u.$$
Here the subscript 2 in the notation for surface $\mathcal Y_2$  points to the fact that this will 
be the variety appearing in the mirror for the orbit of $\mathfrak{sl}(2)$. \\

The critical points of the potential occur when $(x,y)[u:v]$ takes the values:
$$ p_1\ce (x_1,0)[1:0],  \text{and} \quad p_2\ce (x_2,0)[1:0],$$
where $x_1,x_2$ are solutions of the equation 
$x+1+1/x=0$.
In the affine chart where $u=1$ we have the fiber over $y=0$, given as the curve $Y_o$
  inside the affine chart $\mathbb C^* \times \mathbb C^2$ with coordinates $(x,y,v)$  
cut out by the 2 equations:
$$Y_o\ce \{(x,y,v): y=0= v(x^2+x+1)\}\subset \mathbb C^* \times \mathbb C^2$$ 
which contains 2 double points $p_1$ 
and $p_2$. 
%

Taking $v \mapsto \infty $ we also obtain the point $[0:1] \in \mathbb P^1$, so we see that the critical fibre
is  $Y\ce \{y=0\}$ and has 3 irreducible components 
$Y = Y_0\cup Y_1\cup Y_2$ where
$$Y \ce \left\{\begin{array}{l}
Y_0= \mathbb C^* \times \{[1:0]\}\\
Y_1= \{x_1\} \times  \mathbb P^1\\
Y_2= \{x_2\} \times  \mathbb P^1 .
\end{array}
\right. 
$$

\subsection{Mirror symmetry for $\LG(2)$}
 We need to calculate  the category of $D$-branes  $DB(\mathcal Y_2,g)$ of
 the Landau--Ginzburg model $(\mathcal Y_2,g)$.
Given that  the critical fibre $Y$ has 3 irreducible components, 
following \cite[Example\thinspace 8.2]{KL} we conclude that there are 3 D-branes to be 
considered, $B_i$ corresponding to the components $Y_i$ for $i=0,1,2.$
However, there is a symmetry of the LG-model, given by $x \mapsto x^{-1}$ which 
interchanges $B_1$ and $B_2$ and therefore these are isomorphic in  $DB(\mathcal Y_2,g)$.
We conclude that the category $DB(\mathcal Y_2,g)$ is generated by 2 objects, namely the 2 branes 
$B_0$ and $B_1$. To compute the morphisms we will pass to the category 
$D_{sg}(Y)$ of singularities of $\mathcal Y_2$ which is equivalent to the category of the pair
$DB(\mathcal Y_2,g)$.

Let $\pi\colon \widetilde {\mathcal Y_2} \rightarrow \mathcal Y_2$ 
denote the smooth surface obtained by blowing up $\mathcal Y_2$ at the double points $p_i$ for $i=1,2$ 
with $E_i$ the corresponding  exceptional divisors, hence  $E_i\simeq \mathbb P^1$.
Denote by $\mathcal F$ the sheaf on $\widetilde {\mathcal Y_2}$ supported on  $E_1$  with rank 1 
and degree $-3$ on it, that is, it
satisfies
$\mathcal F\vert_{E_1} = \mathcal O_{E_1}(-3) \simeq \mathcal O_{\mathbb P^1}(-3)$   and it is   extended by zero
to the complement of $E_1$ in $\widetilde{\mathcal Y_2}$.

Let $\mathcal G$ denote the sheaf on $\widetilde{\mathcal Y_2}$ supported on $\pi^{-1}( Y_1)$ (more precisely on the proper transform 
$\widetilde{Y_1}$ of $Y_1$ in $\widetilde{S}$ and on the divisor $E_1$)
 and of degree $-1$  over each curve, that is 
 $\mathcal G\vert_{\widetilde{Y_1}}\simeq \mathcal O_{\mathbb P^1}(-1) \simeq \mathcal G\vert_{E_1}$
  supported on $\pi^{-1}(Y_1)$ and extended by zero to the complement of $\widetilde Y_1$ in $\widetilde {\mathcal Y_2}$.

\begin{lemma}\label{obs}Set $\mathcal L_0\ce \pi_*\mathcal F$ and 
$\mathcal L_1\ce \pi_*\mathcal G$.
The equivalence $DB(\mathcal Y_2,g)$ to $D_{sg}(Y)$ is given at the 
level of objects by $B_0 \mapsto \mathcal L_0$ and $B_1 \mapsto \mathcal L_1$.
\end{lemma}

\begin{proof} Lemma \ref{genODP} shows that the restriction of $\mathcal L_0$ generates $D_{sg}$ at $p_i$ for 
each of the singular points, and note that the 2 points  appear together  connected by $B_0$.
$\mathcal G$ generates the derived category of coherent sheaves on $Y_1$ and since the support of $\mathcal L_1$
contains the critical point $p_1$ it  is not a perfect sheaf on $\mathcal Y_2$, hence it is 
 nontrivial in $D_{sg}(Y)$.\end{proof}

\begin{lemma}\label{morfs} The morphisms are:
$$\left\{
\begin{array}{lll}
\Hom(\mathcal L_0,\mathcal L_0) & = \Hom(\mathcal L_1,\mathcal L_1)& =\mathbb Z\\
\Hom(\mathcal L_0,\mathcal L_1) &= \mathbb Z \oplus \mathbb Z[-1]\\
\Hom(\mathcal L_1,\mathcal L_0) & =0.
\end{array}
\right.$$
\end{lemma}

\begin{proof} 
Note that the supports of $\mathcal L_0$ and $\mathcal L_1$ intersect only at the point $p_1$,
 therefore to calculate their  morphisms it suffices to calculate morphisms between $\mathcal F$ and $\mathcal G$.
To calculate their morphisms, note that  $E_1$ is their common support  on $\widetilde {\mathcal Y_2}$ and therefore
their morphisms are trivial elsewhere. We have that
$\mathcal {\bf Hom}(\mathcal L_0,\mathcal L_1)$ 
come from their restriction to $E_1 $,
giving:

$
\mathbf{Hom}(  \mathcal F\vert_{E_1}, \mathcal G\vert_{E_1})= 
 \mathbf{Hom}( \mathcal O_{E_1}(-1) , \mathcal O_{E_1}(-3) )  = \Hom( \mathcal O_{E_1}(-1) , \mathcal O_{E_1}(-3)) \oplus $
 
$\Ext^1( \mathcal O_{E_1}(-1), \mathcal O(-3) ) =
H^0(\mathcal O_{\mathbb P^1}(-2)) \oplus H^1(\mathcal O_{\mathbb P^1}(-2))[-1] 
= \mathbb Z \oplus \mathbb Z [-1].$
\end{proof}

We have now obtained:

 \begin{theorem*}[\ref{mirror2}] Homological Mirror Symmetry for $LG(2)$ works as follows:
 \begin{itemize} 
 \item {\bf HMS 1} fails for $LG(2)$, that is, for any (quasi) projective variety $Y$:
 $$ \Fuk(\LG(2))\not\equiv D^b\!\!\Coh(Y).$$ 
\item {\bf HMS 2} holds true  for $LG(2)$, that is, 
 we have an equivalence of categories:
 $$\Fuk(\LG(2)) \equiv DB(\mathcal Y_2,g) \equiv D_{sg}(Y).$$
 \end{itemize}
 \end{theorem*}
 
 \begin{proof} The  statement about the nonexistence of  a suitable variety $Y$ 
 to provide the categorical equivalence between $ \Fuk(\LG(2))$ and  $D^b\!\!\Coh(Y)$
 is just a rephrasing of \cite[Thm.\thinspace4.1]{BBGGSM}.
 
 The equivalence between $\Fuk(\LG(2))$ and $D_{sg}(Y) $
is obtained from lemma \ref{obs} at  the level of objects and lemma \ref{morfs} at the level of  morphisms.
 \end{proof}

%
%
%

\begin{figure}
\centering
\begin{tikzpicture}
\shade[top color=Aquamarine, bottom color=Salmon] (4,2) circle (2cm);
\draw (4,2) ellipse (2cm and .5cm);
\draw (4,2) circle (2cm);
\draw[thick,->] (4,-.5) -- (4,-1.5) node[anchor=south west]{$y$};
\draw[thick,->] (6.5,2) -- (7.5,2)node[anchor=south east ]{$2x$};
\node at (4.1,4.23) {N};
\node at (4.1, -.2) {S};
\draw[fill] (4,4) circle (.05cm);
\draw[fill] (4,0) circle (.05cm);
\draw[fill] (8,4) circle (.05cm);
\draw[fill] (8,0) circle (.05cm);
\draw[fill] (4,-2) circle (.05cm);
\draw[fill] (8,0) -- (8,4);
\draw[fill] (2,-2) -- (6,-2);
\end{tikzpicture}
\caption{$\LG_2$ and its mirror}
\label{sphere}
\end{figure}

Thus, the derived category obtained in Lemma\thinspace\ref{morfs} 
indeed coincides with the Fukaya category $\Fuk(\LG(2))$ 
described in \cite[Thm.\thinspace 3.1]{BBGGSM}.
We observe that from the viewpoint of configuration of critical points the 
two LG models behave like as if $(\mathcal Y_2,g)$ had been obtained from
$\LG(2)$ by a $90^o$ rotation: the former configuration of 2 critical points on a single fibre
can be achieved by changing the potential on the latter to a direction perpendicular 
to the original one.

\section{The Landau--Ginzburg model $\LG(3)$}

The Landau--Ginzburg model  $\LG(3)$ for the minimal adjoint orbit $\mathcal O_3$ of $\mathfrak{sl}(3)$ 
is described in detail in \cite[Sec.\thinspace 4.1]{BGRSM}. 
The adjoint orbit
$\mathcal O_3 $ compactifies  to the product  $ \mathbb P^2\times \mathbb P^2$
where it is embedded as the open orbit of the diagonal action of $\mathrm{SL}(3,\mathbb C)$.
 This action has as a closed orbit the divisor $D_1= \mathbb F(1,2)$
and we have $ \mathbb P^2\times \mathbb P^2\setminus \mathbb F(1,2) \simeq \mathcal O_3 $. 

So, we study the pair $(X,D_1) = (\mathbb P^2\times \mathbb P^2, \mathbb F(1,2))$ with a potential $f_H$ 
obtained by  choosing a regular element $H=\Diag(\lambda_1,\lambda_2,\lambda_3)$.

The rational map $R_H$ extending   $f_H$ to $ \mathbb P^2\times \mathbb P^2$ is 
 described in \cite{BGGSM} as
  \begin{equation}\label{pot3}
R_H[(x_1:x_2:x_3),(y_1:y_2:y_3)]= 
\frac{\lambda_1 x_1y_1+\lambda_2x_2y_2+\lambda_3x_3y_3}{x_1y_1+x_2y_2+x_3y_3}\text{.}
\end{equation}
The denominator vanishes precisely over the flag manifold, that is, over
\begin{equation*} \label{flag} \mathbb F(1,2) = \{ [(x_1:x_2:x_3),(y_1:y_2:y_3)]  \in \mathbb P^2 \times  \mathbb P^2;  x_1y_1+x_2y_2+x_3y_3=0 \},\end{equation*}
  and the indeterminacy locus $\mathcal I$  of $R_H$, is the divisor  in $\mathbb F(1,2)$ cut out by 
  \begin{equation*} \label{ind} \mathcal I = \{ [(x_1:x_2:x_3),(y_1:y_2:y_3)]  \in \mathbb F(1,2);  \lambda_1x_1y_1+\lambda_2x_2y_2+\lambda_3x_3y_3=0 \}. \end{equation*}

\subsection{Monodromy and vanishing cycles}
The potential of $\LG(3)$ has 3 critical  points, happening at the 3 Weyl images of $\Diag(2,-1,-1)$. 
In the compactification they correspond to the 3 coordinate points $(e_i,e_i) \in \mathbb P^2\times \mathbb P^2$,
where $e_1= [1:0:0]$ etc. 
Since $\LG(3)$ is  a symplectic Lefschetz fibration, we know that the monodromy around each of the singular 
fibres is a classical Dehn twist. 

We calculate the possible monodromy matrices 
fitting into this situation when we choose a regular fibre $Y_b$. 
By \cite[Cor.\thinspace 3.4]{GGSM1} we know that $H_3(Y_b) = \mathbb Z \oplus \mathbb Z$.
We can write  down the middle homology 
of the regular fibre as being generated by the vanishing cycles represented in 
homology as  $\sigma_1 = (1,0)$ and $\sigma_2= (0,1)$. 

  By definition of 
Lefschetz fibration,  
around each critical point
we can choose coordinates to write 
 the potential  as 
$z_1^2+z_2^2+z_3^2+z_4^2$, and  each single monodromy 
 may be expressed up to a change of coordinates as the action of 
$T={\tiny  \left(\begin{matrix} 1 & 1 \\
                                             0 & 1
               \end{matrix}\right)},$ 
 with         the vanishing cycle 
being   fixed by the Dehn twist.
%
Given that the potential was compactified to a map onto $\mathbb P^1$, we are looking for 3 monodromy 
 matrices satisfying 
$T_3T_2T_1=I$, each representing a Dehn twist. 
In further generality this composition ought to equal the monodromy around the fibre at infinity;
however, in our case, as proved in \cite[Lem.\thinspace 17]{BGRSM},  there is no critical point at infinity.
Let
$T_1 \ce T$
 describe the  monodromy  around the first critical fibre,
thus  fixing  vectors of the form     $(a,0)$ in $H^3(Y_b)= \mathbb Z \oplus \mathbb Z$.             
     The natural choice for $T_2$ is a Dehn twist that fixes the vectors $(0,b)$,  represented as 
     $T_2= T_1^t$, the transpose of $T_1$.
Then $T_3 = (T_2T_1)^{-1}$
fixes only $(0,0)$, which is geometrically clear, since it must invert both
actions of $T_1$ and $T_2$. But the vanishing cycle corresponding to $T_3$ can not
be zero, so the only possibility left is that the vanishing cycle corresponding to $T_3$ occurs at infinity,
at the limit  of the unstable manifold corresponding to the third critical point in the open orbit 
$\mathcal O_3$. Hence, a vanishing cycle is  placed at the compactifying
divisor $D_1$.
We conclude that 2 of the critical points have their corresponding vanishing cycles on the regular fibre, 
whereas the third critical point corresponds to a vanishing cycle at infinity. We state this fact 
as a lemma. 

\begin{lemma}\label{vanishc3}
The vanishing cycles of $\LG(3)$ can be depicted as the 2 generators of the 
middle homology of the regular fibre, 
together with one sphere at infinity.  \end{lemma}

\begin{figure}
\includegraphics[height=6.6cm]{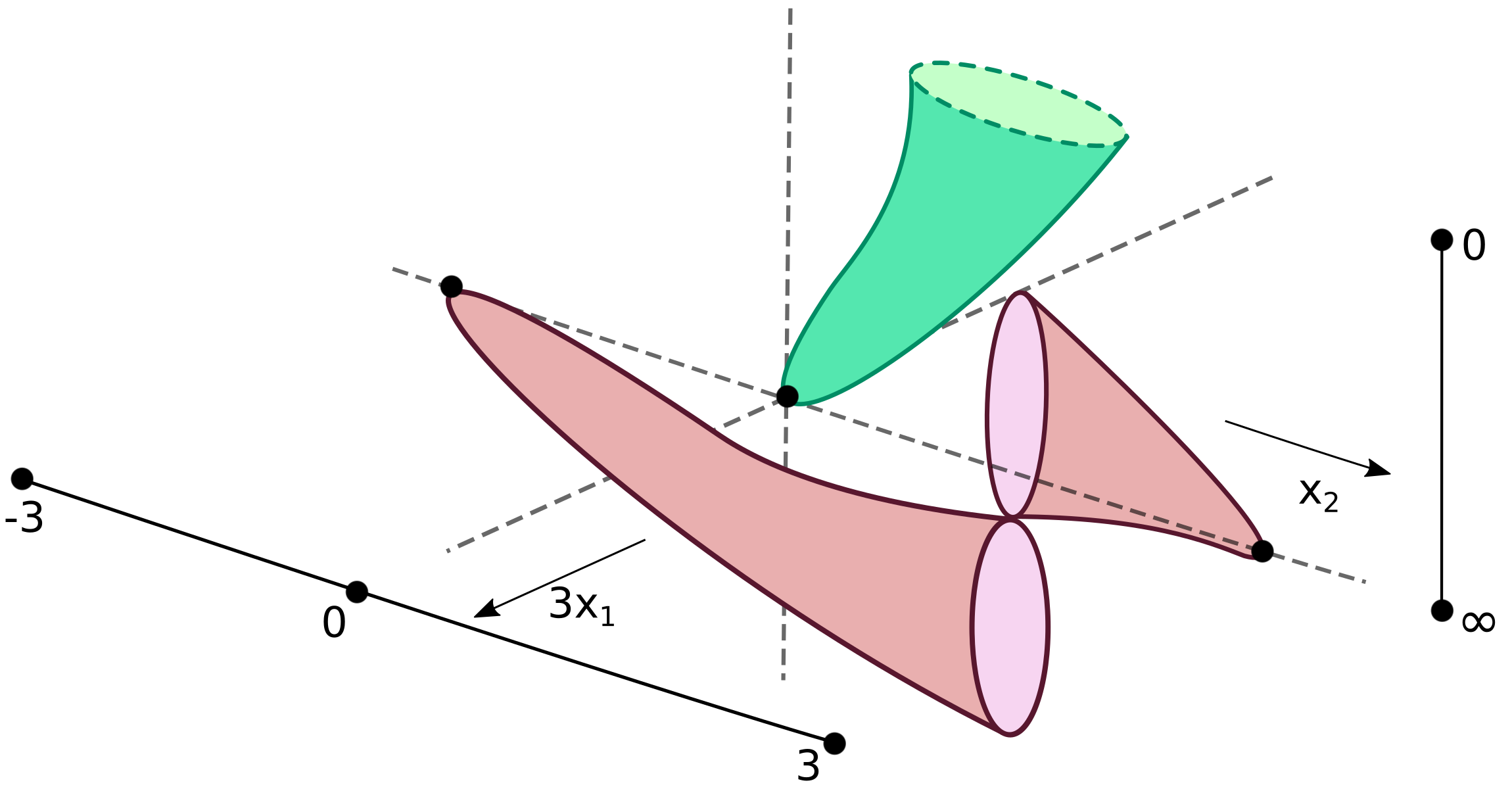}
\caption{$\LG(3)$ and its mirror}
\label{fig2}
\end{figure}

By the end of the following section, 
we will verify
that   the mirror  reverses this  behaviour,  having one critical point over the fibre 
at zero and 2 others at infinity.

\subsection{The intrinsic mirror of $\LG(3)$}\label{3}

The minimal adjoint $\mathcal O_3$ orbit is compactified to  $\mathbb P^2\times\mathbb P^2$ with
fibre  at infinity  $D_1 \simeq  \mathbb F(1,2)$ and  potential  of bidegree $(1,1)$ coming from \ref{pot3}.
We take a second divisor $D_2$ as another 
fibre of the potential, so that $D_1\cap D_2$ is the indeterminacy
locus. 
To get to the  log Calabi--Yau situation, we need additional divisors $D_3$, $D_4$
of bidegree $(1,0)$ and $(0,1)$ respectively, so that 
$D=\sum D_i$ is bidegree $(3,3)$, the anticanonical class. Their intersections satisfy:

$D_1\cap D_3$ and $D_2\cap D_3$  are of type $(1,1)$ in $\mathbb P^1\times \mathbb P^2$,

$D_1\cap D_4$ and $D_2\cap D_4$  are of type $(1,1)$ in $\mathbb P^2\times \mathbb P^1$,

$D_3\cap D_4 =  \mathbb P^1 \times \mathbb P^1$ and $D_1\cap D_2$  is a connected surface,

$D_1\cap D_3 \cap  D_4$ and $D_2\cap D_3\cap D_4$ are  of type $(1,1)$ in $\mathbb P^1\times \mathbb P^2$,

$D_1 \cap D_2 \cap D_3$ is an intersection of 2  type $(1,1)$ divisors in $\mathbb P^1\times \mathbb P^2$, 

 $D_1\cap D_2 \cap D_3 \cap D_4$ consists of 2 connected components  $(1,1) \cap (1,1)$ 
in $\mathbb P^1\times \mathbb P^1.$ 
This gives  a pillow-type  picture   similar to the one 
in  \cite[Example\thinspace2.10]{GS} and our argument follows 
along the same lines.
Using the Gross--Siebert recipe, we  introduce functions
corresponding to the barycenters of the two 4-dimensional cones. Writing 
$D=D_1+D_2+D_3+D_4$ and variables  
$x_1,x_2,x_3,x_4,y_1,y_2$, 
the corresponding equations for the mirror of $(\mathbb P^2\times\mathbb P^2, D) $
are:
\begin{align*}
x_1x_2x_3x_4= {} & y_1+y_2+ t^{(1,0)} x_3 + t^{(0,1)} x_4, \quad \text{and}\\
y_1y_2 = {}  & t^{(1,1)} x_3x_4,
\end{align*}
where $(1,0)$, $(0,1)$ and $(1,1)$ stand for homology classes of divisors  with these degrees. 
In affine space
of dimension 6, with coordinates $x_1,x_2,x_3,x_4,y_1,y_2$
over $\mathbb C[t_1,t_2]$, after calculation of the punctured Gromov--Witten invariants,  
 we obtain: 
\begin{align*}
x_1x_2x_3x_4= {} & y_1+y_2+ t_2 x_3 + t_1 x_4, \quad \text{and}\\
y_1y_2 = {}  & t_1t_2 x_3x_4,
\end{align*}
where $t_1$, $t_2$ may be set to  random values (we really only use here the fact that they are nonzero).
We take as our potential $g= x_2$, corresponding to the first extra divisor $D_2$ that we added.

\subsubsection{The mirror  variety as a hypersurface} 

At first, the ring defining our mirror variety is 
$R=\mathbb C[t_1,t_2][x_1,x_2,x_3,x_4,y_1,y_2]/<r_1,r_2>$ with relations  
$$
r_1\ce x_1x_2x_3x_4=  y_1+y_2+ t_2 x_3 + t_1 x_4, \quad r_2\ce y_1y_2 =  t_1t_2 x_3x_4.$$
Observing that $R$ is isomorphic to 
$\mathbb C[t_1,t_2][x_1,x_2,x_3,x_4,y_1]/<r>$ with the single relation 
$r\ce  y_1(x_1x_2x_3x_4-  y_1- t_2 x_3 - t_1 x_4) =  t_1t_2 x_3x_4,$
 our mirror candidate variety $\mathcal Y_3$  may described by a single polynomial equation
$$y_1x_1x_2x_3x_4-  t_1t_2 x_3x_4- y_1t_2 x_3 -y_1 t_1 x_4- y_1^2=0 .$$
Equivalently, we take
the variety  defined by
$$y_1x_1x_2x_3x_4= (y_1+t_1x_4)(y_1+t_2x_3),$$
and we  have chosen the potential
\begin{equation}\label{x2} x_2= \frac{t_1t_2 x_3x_4+ y_1t_2 x_3 +y_1 t_1 x_4+ y_1^2}{y_1x_1x_3x_4}.\end{equation}
In further generality, we could have considered $x_2+x_3+x_4$ corresponding to all 
divisors added after compactifying, in order to arrive at the log geometry setup, but we will see that our simpler choice is enough 
for our purposes of finding a mirror.

\begin{remark} Both choices of potential $x_2$ or $x_2+x_3+x_4$ 
lead to the same category of singularities. Indeed, suppose that instead of choosing the potential to be $x_2$ we 
  had chosen the potential to be 
$$\phi = x_2+x_3+x_4.$$
Then, using \eqref{x2} we could rewrite it as
$$\phi =  \frac{t_1t_2 x_3x_4+ y_1t_2 x_3 +y_1 t_1 x_4+ y_1^2}{y_1x_1x_3x_4} +x_3+x_4$$ 
or equivalently
$$\phi =  \frac{ (y_1+t_1x_4)(y_1+t_2x_3)+(x_3 + x_4)(y_1x_1x_3x_4)}{y_1x_1x_3x_4} .$$     
In order to homogenize  $\phi$ (keeping in mind the $t_i$'s are constants) 
   we take an extra variable $y_2$ and write
$$\phi =  \frac{ (y_1+t_1x_4)(y_1+t_2x_3)y_2^3+(x_3 + x_4)(y_1x_1x_3x_4)}{y_1x_1x_3x_4y_2}. $$     

We would then arrive at an expression analogous to the one in Lemma \ref{homog}, with $R=[\phi:1]$
in place of $[x_2:1]$. The corresponding rational map would have indeterminacy locus
  $$(y_1+t_1x_4)(y_1+t_2x_3)y_2^3+(x_3 + x_4)(y_1x_1x_3x_4) = 0 = y_1x_1x_3x_4y_2.$$  
Since  the r.h.s. must equal  zero, one sees clearly that there are 2 cases:

a)  $ y_1x_1x_3x_4= 0 $ which implies    $(y_1+t_1x_4)(y_1+t_2x_3)y_2^3 = 0 $. Therefore here 
we obtain exactly the same indeterminacy locus as in \eqref{indl}. 

b) $y_2=0$ which implies $(x_3 + x_4)(y_1x_1x_3x_4) = 0$ and either the second term on the l.h.s. is zero 
and we fall back on case a), or else $x_3+x_4=0$ and we are in the same situation as choosing just $x_2$ as 
the potential. 

\end{remark}

\subsubsection{Singularities of the  mirror $(\mathcal Y_3, x_2)$ and rational extension of the potential}

Since the   $t_i$'s are constants,
we discuss the polynomials  in $\mathbb C[x_1,x_2,x_3,x_4,y_1]$ giving the variety cut  $\mathcal Y_3\subset \mathbb C^5$ out by the equation 
$$y_1x_1x_2x_3x_4= (y_1+t_1x_4)(y_1+t_2x_3).$$

\begin{lemma}\label{sings3} The singularities of the variety $\mathcal Y_3$ 
occur when either 
$y_1=x_3=x_4=0$ or else 
$x_1=x_2=y_1+t_2x_3=y_1+t_1x_4=0.$
\end{lemma}

\begin{proof}
The variety is given by  $$p=-x_1x_2x_3x_4y_1+(y_1+t_1x_4)(y_1+t_2x_3)$$
and the partials are
$$ \left\{
\begin{array}{l}
\frac{\partial p}{\partial x_1} = -x_2x_3x_4y_1\\
\frac{\partial p}{\partial x_2} = -x_1x_3x_4y_1\\
\frac{\partial p}{\partial x_3} = -x_1x_2x_4y_1+t_2y_1+t_1t_2x_4\\
\frac{\partial p}{\partial x_4} = -x_1x_2x_3y_1+t_1y_1+t_1t_2x_3\\
\frac{\partial p}{\partial y_1} = -x_1x_2x_3x_4+2y_1+t_1x_4+t_2x_3.\\
\end{array}\right.$$
%
Vanishing of the equations occurs in 2 separate cases:

Case (1) $y_1=0$ which leads  to 
$ y_1=x_3=x_4=0.$

Case (2) $x_1=x_2=0$ which leads  to 
%
$ x_1=x_2=y_1+t_2x_3=y_1+t_1x_4=0.$
%
%
%
%
\end{proof}

Now, we study the potential determined by $x_2$ over $\mathcal Y_3$. 

\begin{lemma}\label{rational-potential3} Let $\mathcal Y_3'= \mathcal Y_3 \cap \{y_1\neq 0\}\cap \{x_1\neq 0\} \cap  \{x_3\neq 0\} \cap  \{x_4\neq 0\} $, then 
the potential may be regarded as a smooth map  $x_2\colon \mathcal Y_3'\rightarrow \mathbb C$ given as a quotient 
$$x_2= \frac{(y_1+t_1x_4)(y_1+t_2x_3)}{y_1x_1x_3x_4},$$
and the only finite critical value of the potential  $x_2\colon \mathcal Y_3'\rightarrow \mathbb C$ is $x_2=0$.
\end{lemma}

\begin{proof} 
If the left hand side of 
$y_1x_1x_2x_3x_4= (y_1+t_1x_4)(y_1+t_2x_3)$
is nonzero, then the equation defines a smooth variety, and since 
$x_1,x_3,x_4,y_1$ are all nonzero on $\mathcal Y_3'$ to obtain a finite critical value of the potential, we must have that $x_2=0$.
(This may also  be checked by standard calculations of the  partial derivatives.)
\end{proof}

Alternatively, we may regard $x_2$ defining as a rational map 
    $ \mathcal Y_3\dashrightarrow \mathbb P^1$, and homogenizing using an extra variable $y_2$ (the $t_i$'s are constants) 
we obtain the following expression:
 $$ R(x_1,x_2,x_3,x_4,y_1,y_2)=    [(t_1t_2 x_3x_4+ y_1  t_2 x_3 +y_1 t_1 x_4+ y_1^2)y_2^2:y_1x_1x_3x_4],$$
which can be viewed  as a rational map $\mathbb P^5 \dasharrow \mathbb P^1$.
When $y_1x_2x_3x_4\neq 0$, we have
 $$ R(x_1,x_2,x_3,x_4,y_1,y_2)=   \left[\frac{(y_1+t_1x_4)(y_1+t_2x_3)y_2^2}{y_1x_1x_3x_4}:1\right]=  \left[x_2:1\right],
$$
showing that $R$ is  indeed an extension of the potential $x_2$.
It is usual to call $\infty\ce [1:0]$ the point at infinity and refer to  the divisor $D_\infty= \{y_1x_1x_3x_4=0\}$ as the fibre at infinity. 
The indeterminacy locus of $R$ is
\begin{equation}\label{indl}\mathcal I = \{(y_1+t_1x_4)(y_1+t_2x_3)y_2^2=0=y_1x_1x_3x_4\} \subset D_\infty.\end{equation}

\subsubsection{The case $t_1=t_2=1$}
To match the notation of the general case in  section \ref{n},
 we rename $x_3= w$ and $x_4=z$ (in fact, it ought to be $w_1$ and $z_1$, but 
we omit the subscripts).
We restate the results of lemmas \ref{sings3} and \ref{rational-potential3} with  $t_1=t_2=1$:

\begin{lemma} \label{homog}
The singularities of the variety $\mathcal Y_3\ce \{x_1x_2y_1zw=(y_1+z)(y_1+w)\}$
occur when either 
$y_1=z=w=0$ or else 
$x_1=x_2=y_1+w=y_1+z=0.$
The potential $x_2$ may be regarded as a smooth map on  ${\mathcal Y}_3'$ given by
$x_2= \frac{ ( y_1 +z)(y_1  +w)}{x_1y_1zw}$, and 
the map
$$R (x_1,x_2,y_1,y_2,z,w)=      \left[\frac{(y_1 +z) (y_1 +w)y_2^2}{x_1y_1zw}:1\right]= [x_2:1]$$
 is a rational extension of $x_2$ to the compactification. 
  \end{lemma}
 Considering  $\mathbb P^5$
 with coordinates $([x_1:x_2:y_1:y_2:z:w])$,
we may write $\mathcal Y_3$ as the hypersurface  
(contained inside the affine chart $y_2=1$)
 cut out  by the equation
 $$x_1x_2y_1zw=(y_1+z)(y_1+w).$$
There are 2 special points of the potential:
\begin{itemize}
\item  $R=[0:1]$, we call this the zero of the potential $x_2$, 
 happening when $(y_1+z)(y_1+w)y_2^2=0$. Note that  $[0:1]$  is a critical value of the potential which 
contains singularities of the variety ${\mathcal Y}_3 $; and
\item $R=[1:0]$, we call this  the infinity of the potential $x_2$, that is, when  $x_1y_1zw=0$.
Note that $[1:0]$ also contains singularities of ${\mathcal Y}_3$.
\end{itemize}

Now, to study  finite values of the potential,
the value of $x_1$ does not matter at all, as long as it is nonzero. To study the fibre  
 at infinity, we either have $x_1=0$ or else, once again 
 any nonzero value of $x_1$ has no effect. Hence, $x_1$ is only a placeholder 
 marking the fibre at infinity, while $x_2$ is only used to name the map, therefore 
 we may look at the restricted problem of a rational map $\mathbb P^3 \dashrightarrow \mathbb P^1$
$$R'=  [ (y_1 +z)(y_1  + w)y_2: y_1zw]$$ 
with corresponding indeterminacy locus 
$$\mathcal I' = \{ (y_1+z)(y_1+w)y_2=0=y_1zw\}.$$      
We will see that this dimensionally reduced problem gives our desired mirror.

\subsubsection{Extension of the potential to the compactification}
The next step is to consider a blow-up of 
$\mathbb P^3$ along the indeterminacy locus $\mathcal I'$ of the map $R'$.
We  denote by   
${\mathcal Z}_3\ce Bl_{\mathcal I'}(\mathbb P^3)$, the result of this blowing up,
obtained as follows.

Set  $f=(y_1+z)(y_1+w)y_2$ and $g=y_1zw$.
Take $\mathbb P^1$ with  homogeneous coordinates $[t:s]$.  
The pencil $\{tg+sf\}_{t,s \in\mathbb C}$ with base locus $\mathcal I'$ 
induces the rational map $R'\colon \mathbb P^3 \dashrightarrow \mathbb P^1$. 
We denote by ${\mathcal Z}_3$ the
closure of  the graph of the blow-up 
map, so that
\[
\begin{tikzcd}[swap]
       &  {\mathcal Z}_3 \arrow{d}[right]{\pi}\\
	{\mathcal Y}_3 \arrow[hook]{r}[above]{i} \arrow[dashed]{ur}
	\arrow[dashed]{dr}{[x_2:1]}
	& \mathbb{P}^3\arrow[dashed]{d}[right]{R'} \\  
		& \mathbb{P}^1
\end{tikzcd}
\]
and we get a map 
\begin{eqnarray*}
W=  R' \circ \pi \colon {\mathcal Z}_3 \subset  \mathbb P^3\times \mathbb P^1 & \rightarrow &\mathbb P^1\\
(  y_1: y_2:z:w: t: s)&\mapsto&[t:s].
\end{eqnarray*}
Note that if $s\neq 0$ then 
$[t:s]=\left[\frac{t}{s}:1\right]=\left[\frac{f}{g}:1\right]= \left[f:g\right]$
where the middle equality holds since $tg=sf_H$. 
Thus, we have obtain an extension of $f/g$ to the compactification.
We call $\infty\ce [1:0]$ the point at infinity, and refer to  the divisor $D_\infty= \{y_1zw=0\}$ as the fibre at infinity. 

%

\begin{lemma} The potential $W$ has 2 critical values  $[0:1]$ and  $ [1:0]$.
We refer to them  as zero and infinity respectively.
\end{lemma}

We observe that even though here the variety $\mathcal Z_3$ contains singular points, we do not 
consider them individually. Instead, we just take into account 
 the critical fibres of the potential, and then we verify that this choice produces the correct category of singularities
for the mirror.

\subsection{Category of singularities of $\LG^\vee\!(3)$}

Considering $ \mathbb P^3 \times \mathbb P^1$ 
with coordinates $ [y_1: y_2:z:w], [t: s]$, we now analyse the critical fibres of the potential $W$ over  ${\mathcal Z}_3$.
 

\subsubsection{The fibre over $0$ contributes a single object to $D_{sg}(\mathcal Y_3)$}
The critical fibre at $0$ happens when $[t:s]=[0:1]=[f:g]$, that is, when 
$[(y_1 +z) (y_1 +w)y_2:y_1zw]=[0:1]$. 

If $y_2=0$ the equality is satisfied at all points of a  divisor $D_{y_2} = {\mathcal Z}_3 \cap \mathbb P^2\times \mathbb P^1$,
and its contribution to the category of singularities of the compactified model would give us one object, 
which comes from the sheaf $\mathcal F(y_2)$ supported at the points of 
intersection $D_{y_2} \cap \{D_{z_0} +D_{w_0}\}$, where $D_{z_0}$ and $D_{w_0}$ are the divisors that  appear below. 
However,  we are interested only in the critical points of the potential in the mirror variety ${\mathcal Y}_3$ which lives
inside the affine chart $y_2=1$, so that we ignore this object living outside ${\mathcal Y}_3$.

Suppose now $y_2=1$, then we obtain the subset of the fiber over zero  cut out by equations
$(y_1 +z) (y_1 +w)=0$ with singularity in codimension 2 at the
intersection of the 2 divisors $D_{z_0} $ given by $y_1=-z$ and $D_{w_0}$ given by $y_1=-w$.
So, we obtain 2 nonperfect sheaves on $W^{-1}([0:1])$, these are the sheaves  
$\mathcal F(z_0)= {i_z}_*\mathcal O_{D_{z_0}}(1)$
and $\mathcal F(w_0)= {i_w}_*\mathcal O_{D_{w_0}}(1)$ supported 
on $D_{z_0}$ and $D_{w_0}$ included in $W^{-1}([0:1])\cap \mathcal Y_3$
by $i_z$ and $i_w$  respectively.
However, since there is a clear symmetry of this fibre given by exchanging $z$ and $w$ the sheaves
$\mathcal F(z_0)$ and $\mathcal F(w_0)$ are quasi-isomorphic, so
we obtain a unique object of $D_{sg}(W)$, say $\mathcal F(z_0)$ coming from the fibre over $0$.

\subsubsection{The fibre over $\infty$ contributes 2 objects to $D_{sg}(\mathcal Y_3)$}
The critical fibre over  $\infty$ happens when $[t:s]=[1:0]=[f:g]$, that is, when 
$[(y_1 +z) (y_1 +w)y_2:y_1zw]=[1:0]$.
Here we may  choose 3 extra sheaves depending on the 3 variables making $y_1zw=0$, once again, 
there is a clear symmetry of the space taking $z$ to $w$ so that we are left with 2  
new objects $\mathcal F(y_1)$ and $\mathcal F(z)$;
these are rank 1 sheaves supported on $y_1=0$ and $z=0$ respectively.

\begin{lemma}
The category of singularities of $(\mathcal Y_3)$ is generated by 3 objects; namely 
one sheaf supported on the fibre over zero, and 2 sheaves supported on the fibre over infinity.
\end{lemma}

\begin{proof}
The generators are the sheaves $\mathcal F(z_0)$, $\mathcal F(y_1)$ and $\mathcal F(z)$ just defined. 
\end{proof}
         
\section{The Landau--Ginzburg model  $\LG(n+1)$}\label{n}

$\LG(n+1)= (\mathcal O_{n+1},f_H),$ embeds in $\mathbb P^n \times \mathbb P^n$ as the open orbit of the 
diagonal action of $\mathrm{SL}(n+1)$ with complement the diagonal, that is, 
 the flag of lines and $n$-planes in $\mathbb C^{n+1}$. 
 
 \subsection{$\LG(n+1)$ fitted into the log geometry}
We have   $\mathcal O_{n+1} \subset  \mathbb P^n \times \mathbb P^n$ 
with compactifying divisor $\mathbb F(1,n)$.
 We set   $D_1\ce \mathbb F(1,n)$, 
 and chose  a second divisor $D_2$ of type (1,1).
We now describe the intrinsic mirror,  generalising section \ref{3}.

 We write the components of the boundary
divisor as:
$D_1, D_2$   of bidegree (1,1);
$E_1,...,E_{n-1}$ of bidegree (1,0); and
$F_1,...,F_{n-1}$ of bidegree (0,1).
Then we take variables $x_1,x_2$ corresponding to $D_1, D_2$;
$z_1,\dots, z_{n-1}$ corresponding to  $E_1,...,E_{n-1}$; and
$w_1, \dots, w_{n-1}$ corresponding to $F_1,...,F_{n-1}$.

 As in the $n=2$ case, we also need to take additional
variables $y_1,y_2$ which come from the theta functions corresponding to the
contact order information of meeting all divisors transversally; there
are two zero-dimensional strata where this can happen.

\subsection{The intrinsic mirror of $\LG(n+1)$}
The equations coming from these two zero-dimensional strata are
\begin{align*}
x_1x_2 \prod (w_iz_i) {}&= y_1+y_2+t_1 \prod w_i + t_2 \prod z_i\\
y_1y_2 {}&= t_1t_2 \prod(w_iz_i).
\end{align*}

The first equation gives 
$$y_2= x_1x_2 \prod (w_iz_i) - y_1-t_1 \prod w_i - t_2 \prod z_i,$$
and plugging this into the second equation leads to:
$$y_1(x_1x_2 \prod (w_iz_i) - y_1-t_1 \prod w_i - t_2 \prod z_i) = t_1t_2 \prod(w_iz_i).$$
So,  we obtain the affine variety  in  $\mathbb C[x_1,x_2,w_1,...,w_{n-1}, z_1,...,z_{n-1},y_1]$ cut out  by the single equation:
$$y_1^2 + y_1(t_1 \prod w_i + t_2 \prod z_i-x_1x_2 \prod (w_iz_i)) + t_1t_2 \prod(w_iz_i)=0,$$
or equivalently, we have the variety $\mathcal Y_{n+1}$ defined by:
\begin{equation}\label{mirror-recipe}
(y_1 + t_1 \prod w_i) (y_1 + t_2 \prod z_i) =y_1x_1x_2 \prod (w_iz_i).
\end{equation}

Over the variety  $\mathcal Y_{n+1}$, we choose the potential $g=x_2$ corresponding to $D_2$, the first additional  
divisor which was 
added for adapting  the pair $(\mathbb P^n\times \mathbb P^n,\mathbb F(1,n) )$    to the log geometry setup. 
The Landau--Ginzburg model formed by the variety $\mathcal Y_{n+1}$  together with the potential $x_2$ will
give rise to an extended Landau--Ginzburg model $({\mathcal Z}_{n+1},W)$ which we denoted by 
$$\LG^\vee\!(n+1)\ce (\mathcal Y_{n+1}, g=x_2)$$
and will turn out to be our desired mirror (although we only prove this at the level of objects).

\begin{example} The mirror  of $\LG(3)$ is obtained from the general formula \ref{mirror-recipe}
by making $i=1$ so that $w=w_1=x_3$ and $z=z_1=x_4$.
\end{example}

We now must extend the potential to the compactification. As a first step we construct a rational extension. We have
$$g\ce x_2= \frac{(y_1 + t_1 \prod w_i) (y_1 + t_2 \prod z_i)}{ y_1x_1 \prod (w_iz_i)}.$$

Since $x_1\neq 0$ at finite values of the expression for $x_2$, we may reduce the problem to
looking at $x_2$ as  the map
$$x_2= \frac{(y_1 + t_1 \prod w_i) (y_1 + t_2 \prod z_i)}{ y_1 \prod (w_iz_i)},$$
  and then, as a computational tool, we homogenise,  adding an extra variable $y_2$, 
  thus obtaining
  $$\frac{(y_1y_2^{n-1} + t_1 \prod w_i) (y_1y_2^{n-1} + t_2 \prod z_i)}{ y_1 \prod (w_iz_i)}.$$
  
  Equivalently,  
  we now have the rational map on $\mathbb P^{2n-1}$ written with 
  variables $[y_1:y_2:z_1:z_2:\dots:z_{n-1}:w_1:w_2:\dots:w_{n-1}]$ 
  defined by:
  $$R'= [(y_1y_2^{n-1} + t_1 \prod w_i) (y_1y_2^{n-1}+ t_2 \prod z_i)y_2: y_1 \prod (w_iz_i)].$$
  
  \begin{lemma} The rational map $R'\colon \mathbb P^{2n-1}\rightarrow \mathbb P^1$ is an extension of the potential $g\colon \mathcal Y_{n+1} \rightarrow \mathbb C$.
  \end{lemma}
  
 \begin{proof}
 $R'=[x_2:1]$ in the affine chart $y_2=1$, therefore $R'$ is an extension of the potential $g=x_2$ defined over $\mathcal Y_{n+1}$.
 \end{proof}
 
    The map $R'$ has as indeterminacy locus
  $$\mathcal I' = \{(y_1y_2^{n-2} + t_1 \prod w_i) (y_1y_2^{n-2}+ t_2 \prod z_i)y_2= y_1 \prod (w_iz_i)=0\}.$$

\subsubsection{Extension of the potential to the compactification}
The next step is to consider a  blow-up of 
$\mathbb P^{2n-1}$ along the indeterminacy locus $\mathcal I'$ of the map $R'$.
We  denote by   
$\mathcal Z_{n+1}\ce Bl_{\mathcal I'}(\mathbb P^{2n-1})$, the result of this blowing up,
obtained as follows.

Set  $f_1=[(y_1y_2^{n-2} + t_1 \prod w_i) (y_1y_2^{n-2}+ t_2 \prod z_i)y_2$ and $f_2=y_1 \prod (w_iz_i)$.
Take $\mathbb P^1$ with  homogeneous coordinates $[t,s]$.  
The pencil $\{tf_2+sf_1\}_{t,s \in\mathbb C}$ with base locus $\mathcal I'$ 
induces the rational map $R'\colon \mathbb P^{2n-1} \dashrightarrow \mathbb P^1$. 
Call ${\mathcal Z}_{n+1}$ the
closure of  the graph of the blow-up 
map, so that
\[
\begin{tikzcd}[swap]
       &  {\mathcal Z}_{n+1} \arrow{d}[right]{\pi}\\
	{\mathcal Y}_{n+1} \arrow[hook]{r}[above]{i} \arrow[dashed]{ur}
	\arrow[dashed]{dr}{[x_2:1]}
	& \mathbb{P}^{2n-1}\arrow[dashed]{d}[right]{R'} \\  
		& \mathbb{P}^1
\end{tikzcd}
\]
and we get a map 
\begin{eqnarray*}
W=  R' \circ \pi \colon Z \subset  \mathbb P^{2n-1}\times \mathbb P^1 & \rightarrow &\mathbb P^1\\
(  y_1: y_2:z_1:\dots :z_{n-1}:w_1:\dots: w_{n-1}: t: s)&\mapsto&[t:s].
\end{eqnarray*}
Note that if $s\neq 0$ then 
$[t:s]=\left[\frac{t}{s}:1\right]=\left[\frac{f_1}{f_2}:1\right]= \left[f_1:f_2\right]$
where the middle equality holds since $tf_2=sf_1$. 
We call $\infty\ce [1:0]$ the point at infinity, and refer to  the divisor $D_\infty= \{y_1\prod (w_iz_i)=0\}$ as the fibre at infinity.

Thus, we obtain the desired holomorphic extension:

\begin{lemma}\label{tame}
The map $W\colon {\mathcal Z}_{n+1} \rightarrow \mathbb P^1$ is a holomorphic  extension of $R'$. 
 The critical points of $W$ on ${\mathcal Y}_{n+1}' $ come from  critical points of $R$ on ${\mathcal Y}_{n+1}$. 
 \end{lemma}


\begin{lemma} The potential $W$ has 2 critical values  $[0:1]$ and  $ [1:0]$.
We refer to them informally as zero and infinity respectively.
\end{lemma}

We now want to calculate the category of singularities of  $(\mathcal Z_{n+1},W)$.

\subsection{Singular fibres of of $\LG^\vee\!(n+1)$}\label{catsing}

We analyse the critical points of the potential  over  ${\mathcal Z}_{n+1} \subset \mathbb P^{2n-1} \times \mathbb P^1$ 
with coordinates $ [y_1:y_2: z_1:...:z_{n-1}: w_1:...:w_{n-1}], [t: s]$, where ${\mathcal Z}_{n+1}$ is given by the equation
$$t(y_1 y_2^{n-2}+ t_1 \prod w_i) (y_1 y_2^{n-2}+ t_2 \prod z_i) y_2=sy_1 \prod (w_iz_i).$$

\subsubsection{The fibre over $0$}\label{fibre0} This happens when $[t:s]=[0:1]$.
If $y_2=0$ the equality is satisfied at all points of a  divisor 
$D_y = {\mathcal Z}_{n+1} \cap( \mathbb P^{2n-1}\times \mathbb P^1)$,
and its contribution to the category of singularities would give us one object, 
 coming from the sheaf $\mathcal F(y_2)$ supported at the points of 
intersection $D_y \cap \{\sum D_{z_i}+\sum D_{w_i}.\}$), where $D_{z_i}$ and $D_{w_i}$ are the 
divisors that appear below. However, this happens outside of the variety  which interests us, 
 $\mathcal Y_{n+1}$, contained in the affine chart $y_2=1$. So, for our purposes here, this contribution is ignored. 

Assume now $y_2=1$, then we obtain the subset of the fibre over zero intersected with  our variety 
$\mathcal Y_{n+1}$ is cut out by the equation
$$W^{-1}([0:1]) \cap \mathcal Y_{n+1} =\left\{ (y_1 + \prod w_i) (y_1 +  \prod z_i) =0\right\}$$
 which is singular
 when either $y_1=- \prod z_i$ or $y_1=-\prod w_i$. 
Therefore, we obtain 2 nonperfect sheaves   on $W^{-1}([0:1]) \cap \mathcal Y_{n+1}$.
 These are  sheaves $\mathcal F(z)$
supported on $D_z=\left\{y_1+ \prod z_i=0\right\}$
and $\mathcal F(w)$ supported on  and $D_w=\left\{y_1 +  \prod w_i=0\right\}$ 
included in $W^{-1}([0:1])\cap \mathcal Y_{n+1}$
by $i_z$ and $i_w$  respectively.

Since there is a clear symmetry of the variety $\mathcal Y_{n+1}$ obtained by interchanging $z_i \leftrightarrow w_i$
the corresponding nonperfect sheaves are quasi-isomorphic. Hence,  in total the fibre over zero 
contributes with a single  sheaf, say $\mathcal F(z_0)$ corresponding to the 
singularity $y_1=- \prod z_i$. 
 Thus, the fibre over zero contributes one object to $D_{sg}(\mathcal Y_{n+1})$.

\subsubsection{The fibre over $\infty$}\label{fibreinf} This happens when $[t:s]=[1:0]$.
Considering the fibre over $\infty$, that is, when 
$$ y_1 w_1w_2\dots w_{n-1}z_1z_2\dots z_{n-1}=0,$$  we see that
the fibre at infinity has $2n-1$ irreducible components, with  sheaves supported
on their corresponding  individual components and of degree 1 over their support, as follows:

\begin{itemize}
\item $\mathcal F(w_i)$ supported on $w_i=0$,
\item $\mathcal F(z_i)$ supported on $z_i=0$ for $i=1, \dots, n-1$,
\item $\mathcal F(z_n)\ce \mathcal F(y_1)$ supported on $y_1=0$.
\end{itemize}
Given the symmetry between $z$'s and $w$'s we have that for a fixed $i$ the sheaves  $\mathcal F(w_i) \simeq \mathcal F(z_i)$ are isomorphic;
 so that 
we obtain a total contribution $n$ objects from the fibre at infinity, $\mathcal F(z_i)$ with $i=1, \dots n$.

We have now obtained all the $n+1$ generators  we needed to describe 
 the category of singularities $D_{sg}(\mathcal Z_{n+1})$ of the Landau--Ginzburg model $(\mathcal Z_{n+1},W) $, 
which is our proposed mirror of $\LG(n+1)$, they
 are $\mathcal F(z_0), \dots, \mathcal F(z_n)$; where $\mathcal F(z_0)$ comes from the fibre over zero, 
 and the others from the fibre over infinity.

In the following section we clarify some details about the mirror map.

\subsection{The mirror correspondence}

By \cite[Thm.\thinspace 2.2]{GGSM1} we know that  the Landau--Ginzburg model $\LG(n+1)=(\mathcal O(H_0), f_H)$ has 
 $n+1$ singularities  corresponding to the $n+1$ elements 
 $\{wH_0, w\in \mathcal W\}$, where $\mathcal W$ is  the Weyl group  of $\mathrm{SL}(n+1)$.

By \cite[Cor.\thinspace 3.4]{GGSM1} we know that the middle homology of the regular fibre $f_H^{-1}(c)$ 
 of  $\LG(n+1)$ is  
$$H_{\mathrm{mid}}\left(f_H^{-1}(c)\right)= \mathbb Z\oplus \dots \oplus \mathbb Z, \quad n \text{ times}.$$ 


The Weyl group $\mathcal W$ of $SL(n+1)$ is cyclic of order $n+1$. 
Let  $w_0$  be a generator of $\mathcal W$. 
Correspondingly, we denote by $L_i$ the Lagrangian sphere $S^{2n-1}$ that vanishes to the singularity 
 $w_0^i H_0$. Then $L_1, \dots, L_n$ generate the the middle homology of the regular fibre, that is, 
 $$H_{n-1}\left(f_H^{-1}(c)\right)= \mathbb Z<L_1, \dots, L_n>.$$ 
 
The remaining vanishing cycle, the sphere $L_{n+1}$  lives at infinity.
This can be shown by the same reasoning as the one we used to show lemma \ref{vanishc3}, 
namely, if the  cycle $(n+1)$  were a finite
linear combination of $L_1, \dots, L_n$, then the corresponding monodromy matrices ought to 
compose to give the identity, however, $L_{n+1}$ would need to move all elements of the basis 
and this is not consistent with the monodromy of a Dehn twist. So, we have also in the general case
a similar lemma:

\begin{lemma}
The vanishing cycles of  $\LG(n+1)$ can be seen as the $n$ generators $L_i$  of the 
middle homology of the regular fibre, 
together with one sphere at infinity $L_{n+1}$.  \end{lemma}

\begin{lemma} At the level of objects, the mirror map is given by the 1-1 correspondence:
$L_i \Leftrightarrow \mathcal F(z_i), \quad  \text{for } i= 1, \dots, n, \quad 
L_{n+1} \Leftrightarrow \mathcal F(z_0).$
\end{lemma}

In other words, we have obtained:

\begin{theorem*}[\ref{mirrorn}]
The intrinsic mirror symmetry algorithm produces an $\LG$-model $(\mathcal Y_{n+1},g)$
for which
 $$ \Fuk(\LG(n+1)) \simeq D_{sg}(\mathcal Y_{n+1})$$
 is a 1-1 correspondence of objects.
\end{theorem*}

Note that on the symplectic side I have described the vanishing cycles, but 
not the morphisms between them in the Fukaya--Seidel category. This happens because 
even though I know existence of Lagrangian  thimbles taking the given Lagrangian vanishing cycles to  
their corresponding  critical points, I do not know their explicit expressions in the total space, which 
would be necessary for calculating the morphisms. 
Nevertheless, despite not knowing the morphisms on the symplectic (A) side, I can  calculate  morphisms 
on the algebraic (B) side using a simple Macaulay2 code presented in Appendix \ref{thomas} code  \ref{M2}.
We now continue to the description of the   category $D_{sg}({\mathcal Z}_{n+1})$. 

\subsection{Category of singularities of $\LG^\vee\!(n+1)$}

The Orlov category of $\LG^\vee\!(n+1)= ({\mathcal Z}_{n,+1}W)$ is by definition  the sum of the categories of 
singularities  of the fibres over zero and infinity, 
which are disjoint subvarieties of ${\mathcal Z}_{n+1}$. 
The singular fibres  
were calculated in subsections \ref{fibre0} and \ref{fibreinf}. Their singularities categories are generated by 
$\mathcal F(z_0)$ and $\mathcal F(z_1), \dots, \mathcal F(z_n)$ respectively.
To express the morphisms, we first set some notation. 

For calculations corresponding to the fibre at infinity we
set the polynomial  ring $R = \mathbb C[z_1, \dots, z_n]$ with  ideal  $I
= (z_1 \dots z_n) \subset R$  and
quotient ring $S \ce R/I$. We consider modules $J_i$  over $S$ which correspond to the sheaves $\mathcal F(z_i)$ with 
$i=1, \dots n$.  
Hence, we set $S$-ideals
$$ J_i \ce (z_i) \subset S \qquad \text{for}\quad  i=1, \dots n. $$

The required resolution of the module $J_i$ is given by the unbounded, periodic resolution
\[ 0 \leftarrow J_i \xleftarrow{\quad\varepsilon\quad}
   P_0 \cong S^1 \xleftarrow[(z_1 \dots  \hat{z}_i \dots z_n)]{\quad d_0 \quad}
   P_1 \cong S^1 \xleftarrow[(z_i)]{\quad d_1 \quad}
   P_2 \cong S^1 \xleftarrow[(z_1 \dots \hat{z}_i \dots z_n)]{\quad d_2 \quad}
   P_3 \cong S^1 \xleftarrow[(z_i)]{\quad d_3\quad}
   \dotso \text{ .} \]
Then $\Hom_S(P_\bullet, J_j)$ is the cochain complex
\[  0 \longrightarrow
   \Hom_S(P_0 = S^1, J_j) \xrightarrow{f \circ -}
   \Hom_S(P_1 = S^1, J_j) \xrightarrow{g \circ -}
   \Hom_S(P_2 = S^1, J_j) \xrightarrow{f \circ -}
   \dotso \text{ ,} \]
and it represents $R\Hom(J_i, J_j)$ in the derived category.
Since $\Hom_S(S, M) \cong M$ for any $S$-module $M$,
 we have maps
\begin{equation}\label{extsn} 0 \longrightarrow
   (z_j)S \xrightarrow[(z_1\dots \hat{z}_i \dots z_n)]{\quad d_{0*} = f_* \quad }
   (z_j)S \xrightarrow[(z_i)]{\quad d_{1*} = g_* \quad}
   (z_j)S \xrightarrow[(z_1\dots \hat{z}_i\dots z_n)]{\quad d_{2*} = f_* \quad }
   (z_j)S \xrightarrow[(z_i)]{\quad d_{3*} = g_* \quad}
   \dotso \text{ .} \end{equation}
We set the notation  $N_{ij} \ce\ker(d_{2k*}) = (z_i z_j)S$ and $M_{ij} = \ker d_{(2k+1)*} / \im d_{(2k)*}$
for all $k \geq 0$.
We then have:
\[ \Ext^k(S_i, S_j) = \begin{cases}
      N_{ij} & \text{for $k = 0$,} \\
      M_{ij} & \text{for odd $k$, and}     \\
      0 & \text{for even $k$, $k > 0$.} \end{cases} \]
See Appendix \ref{thomas} for further details.      
For $1\leq i<j\leq n$,
the corresponding sheaves 
$\mathcal F(z_1), \dots \mathcal F (z_n)$
therefore satisfy:
$${\mathbf \Hom}(\mathcal F(z_i), \mathcal F(z_j) = N_{ij} \bigoplus_{t=2s+1} M_{ij}[t]$$ 
and 
$${\mathbf \Hom}(\mathcal F(z_0), \mathcal F(z_j))=0.$$

Therefore, we have obtained the full description of the category of singularities, proving: 

\begin{theorem*}[\ref{ort}] The Orlov category of singularities  of the Landau--Ginzburg model $({\mathcal Z}_{n+1},W)$
is generated by the nonperfect shaves   $\mathcal F(z_0), \mathcal F(z_1),\dots, \mathcal F (z_n)$
with  morphisms: 
$${\mathbf \Hom}\left(\mathcal F(z_i), \mathcal F(z_j)\right) = 
\left\{\begin{array}{lllll}
N_{ij} \bigoplus_{t=2s+1} M_{ij}[t]  & \text{if} & i \cdot j \neq 0,\\
0 & \text{if} & i\cdot j =0 \text{.}
\end{array}
\right.$$ 
\end{theorem*}

\section{Acknowledgements} 
The author is grateful to Mark Gross and Dmitry Orlov for 
patiently answering 
many questions and to Thomas  K\"oppe for help with computations in Macaulay2. 
E. Gasparim is a Senior Associate of the 
Abdus Salam International 
Centre for Theoretical Physics, Italy, where the 
final version of this article was written.

\appendix

\small{

\section{Nonperfect sheaves}\label{nonperfect}

Let $Z_2\ce \Tot \mathcal O_{\mathbb P^1} (-2)$ denote the total space of the canonical line bundle on $\mathbb P^1$ and 
let $X_2 $ be the variety obtained from it by contracting the zero section to a point with $\pi\colon Z_2 \rightarrow X_2$
the contraction. Then $X_2$ is a concrete model for the ordinary double point surface singularity. 

Consider 
the line bundle $L = \mathcal O_{Z_2}(-3)$; hence,  
the line bundle on the surface $Z_2$ obtained  from  pulling back $\mathcal O_{\mathbb P^1}(-3)$.
 Set $\mathcal F\ce  \pi_*L$ as the rank 1 sheaf on $X_2$ obtained as direct image.

\begin{lemma}\label{genODP}
  $\mathcal F$ generates $D_{sg}(X_2)$.
 \end{lemma}
 
 \begin{proof}
 \cite[Example\thinspace 2.14]{BGK} computes the resolution of $\mathcal F$ showing  it is infinite and periodic (see details below), 
so that $\mathcal F$ is a coherent sheaf which is not perfect and therefore nonzero in the Orlov category of singularites of $X_2$.
Orlov \cite[Sec.\thinspace 3.3]{Or1} showed that $D_{sg}(X_2)$ has a single generator and $\mathcal F$ is a concrete 
geometric description of it. 
 \end{proof}

We  expand the calculations given in 
 \cite[Example\thinspace 2.14]{BGK} to show that  $\mathcal F$ is not perfect.
 Let $M= \hat{\mathcal F}_0$ be the completion of the stalk at the singular point.
We will describe the structure of $M$ as a module over global 
functions and show that  its resolution  is infinite and periodic.
 In canonical coordinates $Z_2$  is given by charts $U = \{z,u\}$, $V= \{z^{-1},z^2u\}$ 
  and the transition matrix for the bundle $L = \mathcal O(-3)$  is $z^3$.
 The coordinate ring of $X_2$  is $\mathbb C[x,y,z]/(y^2- xw)$, 
 and the lifting map induced by $\pi$  is (on the $U$-chart) $x\mapsto u$ , $y\mapsto zu$ 
 and $w\mapsto z^2u$.
 
 By the Theorem on Formal Functions 
 $M \simeq \lim_{\leftarrow} H^0(\ell_n, L\vert_{\ell_n})$ where  
 $\ell_n$ is the $n$-th infinitesimal neighborhood of $\ell$, and 
using  \cite[Lemma\thinspace 2.5]{BGK} to determine $M$ it suffices to 
compute $\sigma \in H^0(\ell_N, L\vert_{\ell_N})$ for a large enough $N$ and direct computation of \v{C}ech
cohomology gives the expression of such a global section as:
$\sigma = \sum_{r=2}^N\sum_{s=0}^{\lfloor \frac{2r-3}{2}\rfloor} a_{rs}z^su^r.$
It then turns out that all terms of $\sigma$ are generated over 
$$\hat{\mathcal O}_0= \mathbb C[[x,y,z]]/(y^2- xw)$$ by
$\beta_0=u^2=x^2$ and $\beta_1= zu^2=xy$ and we obtain the structure of $M$  as an  $\hat{\mathcal O}_0$-module:
$$M = \hat{\mathcal O}_0[\beta_0,\beta_1]/R^1$$ where 
$R^1$ is the set of relations
$$R^1=\left\{ \begin{array}{l}R^1_1:\beta_0y-\beta_1x\\
R^1_2:\beta_0w-\beta_1y\end{array}
\right..
$$
But then these relations themselves satisfy a second set of relations 
$$R^2=\left\{ \begin{array}{l}R^2_1:R^1_1y-R^1_2x\\
R^2_2:R^1_1w-R^1_2y\end{array}
\right.,
$$
 and so on 
 $$R^n=\left\{ \begin{array}{l}R^n_1:R^{n-1}_1y-R^{n-1}_2x\\
R^n_2:R^{n-1}_1w-R^{n-1}_2y\end{array}
\right..
$$
The corresponding resolution of $M$ therefore is 

$$ \cdots \rightarrow   \hat{\mathcal O}_0 \oplus  \hat{\mathcal O}_0 \stackrel{\tiny \left(\begin{matrix}y & w\\ -x &-y   \end{matrix}\right)}{\longrightarrow}
 \hat{\mathcal O}_0 \oplus  \hat{\mathcal O}_0
\stackrel{\tiny \left(\begin{matrix} y & w\\ -x &-y   \end{matrix}\right)}{\longrightarrow} M \rightarrow 0 \text{.}$$

\section{Orlov category of $\Pi z_i$}\label{thomas}

In the same style as \cite[Sec.\thinspace 3.3]{Or1} we now calculate $D_{sg}(X_0)$ for the affine variety 
$X_0= \mbox{Spec}\left(\mathbb C[z_1, \dots, z_n]/ \Pi z_i\right)$.
Denote by $A$ the algebra $\mathbb C[z_1, \dots, z_n]/ \Pi z_i$. By \cite[Prop.\thinspace 1.23]{Or1}, 
any object of $D_{sg}(X_0)$ comes from a finite dimensional module $M$ over the algebra $A$.

We claim that the modules $M_i= <z_i>$ generated over $A$ by $z_i$ are indecomposable, and  generate 
all modules over $A$. Furthermore, in the correspondence between sheaves on $X_0$ and modules over $A$, 
that is, $\mathcal O_X$-modules =  module over $H^0 (X, \mathcal O_X)$,
they correspond to sheave that are not perfect, since  each of these modules possesses only infinite resolutions. 
This can be verified in Macaulay2 by a simple code, which for  $n=4$ is:
\begin{equation}\label{M2} M2 \quad \text{code}: \end{equation}

Macaulay2, version 1.17.2.1:

i1:k=ZZ/101;

i2:R= k[z1,z2,z3,z4];

i3:I=ideal(z1*z2*z3*z4);

i4:V=R/I;

i5:J1= ideal(z1);

i6:res J1

o6: 
$$ V^1\stackrel{\tiny \left(\begin{matrix} z1\end{matrix}\right)} {\longleftarrow} V^1\stackrel{\tiny \left(\begin{matrix} z2z3z4\end{matrix}\right)}{\longleftarrow}V^1\stackrel{\tiny \left(\begin{matrix} z1\end{matrix}\right)} {\longleftarrow} V^1\stackrel{\tiny \left(\begin{matrix} z2z3z4\end{matrix}\right)}{\longleftarrow}V^1\stackrel{\tiny \left(\begin{matrix} z1\end{matrix}\right)} {\longleftarrow} V^1 $$

The general resolution of $M_i$ takes the form:

$$M_i\stackrel{\tiny \left(\begin{matrix} z_i \end{matrix}\right)} {\longleftarrow} M_i\stackrel{\tiny \left(\begin{matrix} \Pi_{i\neq i}z_i\end{matrix}\right)}
 {\longleftarrow}
M_i\stackrel{\tiny \left(\begin{matrix} z_i \end{matrix}\right)} {\longleftarrow} M_i\stackrel{\tiny \left(\begin{matrix} \Pi_{i\neq i}z_i\end{matrix}\right)}
 {\longleftarrow} M_i \longleftarrow \cdots .
$$

Because any resolution of $M_i$ is infinite the corresponding sheaf will not be perfect.
The following  calculation of  morphisms was written by Thomas K\"oppe.

Let $R = \Bbbk[z_1, z_2, z_3, z_4]$ be a polynomial ring, and $I
= (z_1 z_2 z_3 z_4) \subset R$ a homogeneous ideal, and consider the
quotient ring $S \ce R/I$. We consider modules over $S$. In particular,
for the $S$-ideals
\[ J_1 \ce (z_1) \subset S \qquad \text{and} \qquad J_2 \ce (z_2) \subset S \]
we consider $J_1$ and $J_2$ as $S$-modules in the natural way.

To compute $\Ext^n(A, B)$ in the category of $S$-modules, we can
replace $A$ with a projective resolution
\[ 0 \leftarrow A \xleftarrow{\quad\varepsilon\quad} P_0 \xleftarrow{\quad d_0 \quad} P_1 \xleftarrow{\quad d_1 \quad} \dotso \text{ ,} \]
where the augmentation map $\varepsilon$ is the projection onto the cokernel of $d_0$,
and then obtain $\Ext^n(A, B)$ as the $n^\text{th}$ cohomology of the cochain complex $\Hom(P_\bullet, B)$.

To compute $\Ext_S^n(J_1, J_2)$, we
first  need a resolution of the module $J_1$,
which is given by the unbounded, periodic resolution
\[ 0 \leftarrow J_1 \xleftarrow{\quad\varepsilon\quad}
   P_0 \cong S^1 \xleftarrow[(z_2 z_3 z_4)]{\quad d_0 \quad}
   P_1 \cong S^1 \xleftarrow[(z_1)]{\quad d_1 \quad}
   P_2 \cong S^1 \xleftarrow[(z_2 z_3 z_4)]{\quad d_2 \quad}
   P_3 \cong S^1 \xleftarrow[(z_1)]{\quad d_3\quad}
   \dotso \text{ .} \]
In summary, the resolution $P_\bullet = (P_i, d_i)$ has $P_i \cong S^1$ for all $i = 0, 1, \dotsc$,
and
\[ d_i \colon S^1 \to S^1 \text{ ,\quad} d_i = \begin{cases}f \ce (z_2 z_3 z_4) & \text{for even $i$, and} \\ g \ce (z_1) & \text{for odd $i$,} \end{cases} \]
and $J_1 = \coker(f)$. (The module $J_1$ is not free, since $z_1$ is a zero divisor.)
Then $\Hom_S(P_\bullet, J_2)$ is the cochain complex
\[  0 \longrightarrow
   \Hom_S(P_0 = S^1, J_2) \xrightarrow{f \circ -}
   \Hom_S(P_1 = S^1, J_2) \xrightarrow{g \circ -}
   \Hom_S(P_2 = S^1, J_2) \xrightarrow{f \circ -}
   \dotso \text{ ,} \]
and it represents $R\Hom(J_1, J_2)$ in the derived category.
But note that $\Hom_S(S, M) \cong M$ for any $S$-module $M$,
so we have maps
\begin{equation}\label{exts3} 0 \longrightarrow
   (z_2)S \xrightarrow[(z_2 z_3 z_4)]{\quad d_{0*} = f_* \quad }
   (z_2)S \xrightarrow[(z_1)]{\quad d_{1*} = g_* \quad}
   (z_2)S \xrightarrow[(z_2 z_3 z_4)]{\quad d_{2*} = f_* \quad }
   (z_2)S \xrightarrow[(z_1)]{\quad d_{3*} = g_* \quad}
   \dotso \text{ .} \end{equation}
First, observe that $\ker(d_{2k*}) = (z_1 z_2)S$ for all $k \geq 0$ (multiples of
$z_1 z_2$ are those elements of $J_2$ that vanish when multiplied by $z_2 z_3 z_4$).
In particular, $\Ext^0(J_1, J_2) = \Hom(J_1, J_2) = \ker(d_{0*}) = (z_1 z_2)S$.
Next, $\im(d_{2k + 1}) = (z_1 z_2)S$ for all $k \geq 0$ (the result of multiplying
an element in $J_2$ by $z_1$ is a multiple of $z_1 z_2$). Finally,
\[ \Ext^k(S_1, S_2) = \begin{cases}
      (z_1 z_2)S & \text{for $k = 0$,} \\
      M & \text{for odd $k$, and}     \\
      0 & \text{for even $k$, $k > 0$.} \end{cases} \]
It remains to describe $M = \ker d_{(2k+1)*} / \im d_{(2k)*}$.
We have that he kernel of $d_{(2k+1)*}$ is $(z_2 z_3 z_4)S$, i.e. elements of $J_2$
that vanish when multiplied with $z_1$. The image of $d_{(2k)*}$ consists of multiples of $z_2^2 z_3 z_4$.
Thus the quotient consist of those elements of $(z_2 z_3 z_4)S$ that are linear in $z_2$.
Abstractly, this can be presented as cokernel of the map $[z_2 \ z_1] \colon S^2 \to S^1$,
i.e. as $S / \{ z_2 a + z_1 b \}$: an element of $\Ext^\text{odd}(S_1, S_2)$
is $z_2$ times a function of $(z_3, z_4)$ only.
}

\end{document}